\documentclass{article}
\usepackage[a4paper]{geometry}
\usepackage{amsfonts}
\usepackage{amsthm}
\usepackage{amsmath}
\usepackage{parskip}
\usepackage{mathrsfs}
\usepackage{graphicx}
\usepackage{amssymb}
\usepackage{xtab}
\usepackage[utf8]{inputenc}
\usepackage{mathtools}
\usepackage{caption}
\usepackage{subcaption}
\usepackage{hyperref}
\usepackage{chngcntr}
\usepackage{soul}
\usepackage{xcolor}
\usepackage{cite}
\usepackage{breqn}
\usepackage{bbm}
\usepackage{authblk}
\usepackage{color}

\counterwithin{equation}{section}

\newtheorem{prop}{Proposition}[section]

\newtheorem{lem}[prop]{Lemma}

\newtheorem{thm}[prop]{Theorem}
\newtheorem{cor}[prop]{Corollary}

\theoremstyle{definition}
\newtheorem{defn}[prop]{Definition}

\theoremstyle{remark}
\newtheorem*{rmk}{Remark}

\DeclareMathOperator*{\argmin}{arg\,min}

\DeclareMathOperator{\R}{\mathbb{R}}
\DeclareMathOperator{\E}{\mathbb{E}}
\DeclareMathOperator{\diag}{\textbf{diag}}
\DeclareMathOperator{\1}{\mathbbm{1}}
\DeclareMathOperator{\Tr}{\text{Tr}}

\begin{document}

	\title{From mean field games to the best reply strategy in a stochastic framework}
	\date{}
	\author[1,2]{Matt Barker}
    \affil[1]{Department of Mathematics, Imperial College London}
    \affil[2]{Grantham Institute, Imperial College London}
	\maketitle
    
\begin{abstract}
	This paper builds on the work of Degond, Herty and Liu in \cite{Degond2017} by considering $N$-player stochastic differential games. The control corresponding to a Nash equilibrium of such a game is approximated through model predictive control (MPC) techniques. In the case of a linear quadratic running-cost, considered here, the MPC method is shown to approximate the solution to the control problem by the best reply strategy (BRS) for the running cost. We then compare the MPC approach when taking the mean field limit with the popular mean field game (MFG) strategy. We find that our MPC approach reduces the two coupled PDEs to a single PDE, greatly increasing the simplicity and tractability of the original problem. We give two examples of applications of this approach to previous literature and conclude with future perspectives for this research.
\end{abstract}

\subsubsection*{Key words}
Mean field games, best reply strategy, stochastic differential games, model predictive control, linear-quadratic control

\subsubsection*{AMS subject classification}
Primary: 49N70, 35Q93, 91A13; Secondary: 93C20.

\subsubsection*{Acknowledgements}
Matt Barker would like to thank Pierre Degond, Mirabelle Mu\^uls and Michael Herty for their guidance and suggestions. Matt would also like to thank NERC and the Science and Solutions for a Changing Planet Doctoral Training Partnership at the Grantham Institute at Imperial College London for funding.

\newpage    
    
\section{Introduction}
    
	Mean field game (MFG) models were first proposed by Lasry and Lions \cite{Lasry2007,Lasry2006,Lasry2006a} and simultaneously by Huang, Caines and Malhamé  \cite{Huang2006,Huang2007,Huang2006a,Huang2007a,Huang2006b}. Their research follows from the previous work of Aumann \cite{Aumann1964} and related researchers \cite{Schmeidler1973,Mas-Colell1984} on systems with a continuum of agents. In the last decade, the field has grown considerably with research taking many different directions – from applications \cite{Lachapelle2011,Kizilkale2016,Huang2016}, to existence, uniqueness and regularity \cite{Carmona2016,Bardi2014,Cardaliaguet2015}, and to numerical analysis \cite{Achdou2016,Achdou2012,Achdou2013}. The aim of MFG models is to describe how populations of agents evolve over time due to their strategic interactions. The trajectories of agents are determined through the minimisation of a cost functional over long time horizons. This process of optimisation implicitly assumes agents consider the future evolution of the population for long time periods and are able to continuously change their control.

	For a large number of applications, such as firm behaviour, traffic and pedestrian dynamics, and other human-based optimisation processes, this type of decision-making strategy seems to be different to reality. It would be more likely that in such applications, agents fix their control over a short period of time, evolve their position and then update the control. In recent years, another model of agent interaction has been developed – that of the ‘best reply strategy’ (BRS). The BRS was used in \cite{Degond2014b} to describe agents whose strategies evolved on a faster time scale than their social configuration. It was applied, in \cite{Degond2014} and \cite{Degond2014a} respectively, to the evolution of wealth in conservative and non-conservative economies. In a later paper by Degond, Herty and Liu \cite{Degond2017}, it was shown that the BRS could be related to a rescaled mean field game model in the case of deterministic dynamics.\footnote{It should be noted that the rescaling doesn't necessarily preserve the nash equilibrium of the original MFG model, so while this relation is important in understanding the transition from MFG to BRS, the resulting BRS dynamics can't be expected to approximate the nash equilibria of the original MFG model in general.}

	This relationship was described by a discretisation of the MFG, using a method of control known as model predictive control (MPC) or receding horizon control. The MPC method is detailed in \cite{Herty2017PerformanceDynamics} and the references therein. In this paper we extend the work in \cite{Degond2017} to include idiosyncratic noise in individual dynamics and a terminal cost in the optimisation functional. In a similar manner to \cite{Degond2017}, figure \ref{fig:square} can explain how we relate the BRS to MFGs via the MPC approach. The description of the dynamics begins with an $N$-player stochastic differential game (top left box in figure \ref{fig:square}). We then have a choice of either using MPC to take us to the BRS for the $N$-player games (top right box) or we can take the number players to infinity to obtain the MFG (bottom left box). The two methods then converge to the mean field BRS (bottom right box) either by using MPC from the MFG or taking the limit as the number of players go to infinity from the BRS for the $N$-player game.

  \begin{figure} [h]
      \centering
      \includegraphics[width=0.5\textwidth]{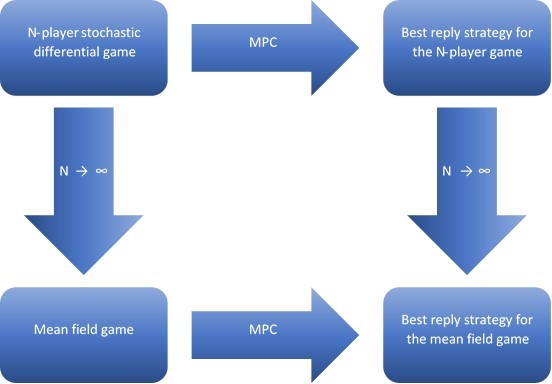}
      \caption{Schematic diagram describing the links between $N$-player games, maen field games and model predictive control.}
      \label{fig:square}
  \end{figure}

	This paper is laid out as follows. In Section 2 we describe the setup of the $N$-player stochastic differential game; we use two methods to explain how the MPC approach results in the BRS. We then take a mean field limit of the controlled $N$-player dynamics to obtain the mean field BRS. In Section 3 we develop the MFG related to the $N$-player stochastic differential game by taking the number of players to infinity. We then show how the MPC approach can be used to recover the mean field BRS from the MFG. In Section 4 we take a number of examples from the MFG and BRS literature and redesign them in the paradigm of this paper as examples for both how this approach might simplify numerical calculations and how the current BRS literature can be considered under the MFG paradigm. In Section 5 we summarise our results and explain future directions for research, both in terms of applications and theory. Finally, in the Appendix we define and discuss various notions of differentiability in the space of measures.

\section{Model predictive control of large interacting systems and the best reply strategy}

    \subsection{An N-player stochastic differential game} \label{sec:MPC}
    
    	Consider $N$ players labelled  by $i = 1, \ldots, N$. Each player has a state $X_i \in \R^d$, which they are controlling over a time horizon of $[0,T]$. Throughout the entire paper, $N$ denotes the number of players in the game, $d$ denotes the dimension of the state space, $i,j \in \{1, \ldots N\}$ denotes the $i$th or $j$th player, and finally $k \in \{1, \ldots d\}$ denotes the $k$th component of a player's state. We denote by $X_{j,k} \in \R$ the $k$th component of player $j$'s state, by $X_j = (X_{j,k})_k \in \R^d$ player j's state, by $X = (X_j)_j = ((X_{j,k})_k)_j \in \R^{Nd}$ the total state of the system and by $X_{-i} = (X_j)_{j \neq i} \in \R^{(N-1)d}$ the state of the system without player $i$. This convention will be similarly used for other variables or functions, unless otherwise stated.

      	We assume that each player's dynamics are influenced by the state of the entire system. This influence is interpreted as a function $f_i^{(N)}: \R^{Nd} \to \R^d$. We also assume that each player can control their dynamics with a control $u_i:[0,T] \to \R^d$. Finally, we include a randomness to the dynamics depending on time and a player's own position, this is given by the function $\sigma:[0,T] \times \R^d \to \diag(\R^d)$, where $\diag(\R^d)$ is the set of diagonal $d \times d$ matrices. Player $i$'s dynamics can therefore be summarised by  

      	\begin{equation}
          	\begin{aligned}
              	&dX_i(t) = (f_i^{(N)}(X(t)) + u_i(t))dt + \sigma(t,X_i(t))dB_i(t) \\
              	&X_i(0) = X_i^0 .
          	\end{aligned}
          	\label{eq:dynamics}
      	\end{equation}
     
      	Here, $B_i(t)$ are independent $d$-dimensional Wiener processes with $t \in [0,T]$ and the initial condition, $X_i^0$, are given iid random variables for all $i = 1, \ldots, N$. We assume that each player wants to minimise its own objective functional.
        
        \vspace{\parskip}
        
        \begin{defn} \label{def:BRS}
        	Given a set of admissable controls $\mathcal{A}$, from which players choose their strategies, we define the objective functional, $J_i: \mathcal{A}^N \to \R$ for player $i$ as the cost of player $i$'s trajectory when controls $u = (u_i)_i:[0,T] \to \R^{Nd}$ are used, i.e.

          	\begin{equation}
              	J_i(u) = \E \left[ \int_0^T \left( \frac{\alpha_i(s)}{2} |u_i(s)|^2 + h_i^{(N)}(X(s)) \right) ds  + g_i^{(N)}(X(T))\right]. \label{eq:objfun}
          	\end{equation}
		\end{defn}
            
		The objective functional's constituent parts are: the running cost of being in position $x$, given by $h_i^{(N)}(x)$; the terminal cost of ending up in position $x$ at the end of the control horizon, given by $g_i^{(N)}(x)$; and the running cost $\frac{\alpha_i(s)}{2} |u_i(s)|^2$ which is used to penalise the size of the control function. Therefore, each player is choosing a strategy to minimise this objective functional, as in \cite{Lasry2007}, this corresponds to a Nash equilibrium where no player can reduce their cost any further by changing their strategy only. We denote the optimal strategy for player $i$ by $u_i^*$. It is then given by the following minimisation problem
        
        \begin{equation}
        	J_i(u^*) = \min_{u_i \in \mathcal{A}} J_i(u_i, u_{-i}^*) \qquad \forall i = 1, \ldots, N. \label{eq:opt}
		\end{equation}
        
         We assume agents can choose controls from a certain set $\mathcal{A}$ of admissable controls. This set usually consists of progressively measurable controls with constraints on their smoothness and integrability. In the finite-player case, we assume that we are using closed-loop controls i.e. that for each $i$ there is a deterministic function $\phi_i:[0,T] \times \R^{Nd} \times \R^{Nd} \to \R$ such that
    
    \begin{equation}
    	u_i(t) = \phi_i(t,X_0,X_t).
    \end{equation}
         
        For a discussion on different sets of admissable controls as well as requirements on the various functions $f_i^{(N)}, h_i^{(N)},g_i^{(N)}, \alpha_i$, see Chapter 2 of \cite{Delarue2018}. As an example of such results, we can ensure existence of solutions to the SDE/optimisation problem \eqref{eq:dynamics} and \eqref{eq:opt} if the following hold:
        \begin{itemize}
             \item $f_i^{(N)}$ is Lipschitz, locally bounded and continuously differentiable and and $f_i^{(N)}(X(t))$ is $L^2$ bounded in time for any control $u_i$.
             \item $\sigma$ is Lipschitz, locally bounded and continuously differentiable in $x$ and $\sigma(t,X_i(t))$ is $L^2$ bounded in time for any control $u_i$.
             \item $g_i^{(N)}$ is locally bounded, continuously differentiable with a derivative that has at most linear growth and convex.
             \item $h_i^{(N)}$ is locally bounded and continuously differentiable with a derivative that has at most linear growth
             \item $\alpha_i(t) > 0$.
        \end{itemize} 
        These requirements guarantee the convexity of each optimisation problem and the existence of solutions to the SDE when using the optimal strategy.
    
    We also introduce the following definition of player $i$'s value functional. This is closely related to the objective functional \eqref{eq:objfun} and is used in the description of the optimal strategies
    
    \vspace{\parskip}

    \begin{defn}
       	We define the value functional, $V_i = V_i(t,x):[0,T] \times \R^{N d} \to \R$ for player $i$ as the cost of player $i$'s trajectory from time $t$ to $T$, with agents starting at position $x$, using their optimal controls. i.e.

       	\begin{equation}
           	V_i(t,x) = \E \left[ \int_t^T \left( \frac{\alpha_i(s)}{2} |u_i^*(s)|^2 + h_i^{(N)}(X(s)) \right) ds  + g_i^{(N)}(X(T))\right], \label{eq:obfun}
       	\end{equation}

       	where $X(s)$ solves \eqref{eq:dynamics} with control $u^*$ and initial condition $X(t) = x$. Note that in this definition $x \in \R^{Nd}$ is a deterministic initial condition. 
	\end{defn}
	
	As well as the optimal strategy $u^*_i$, this paper is interested in the potentially sub-optimal strategy, denoted by $\bar{u}_i$, know as the best reply strategy (BRS). The reason for interest in this is its ease of computation and that in certain situations it may be a better modelling paradigm for agent behaviour than a mean field game.
	
	\vspace{\parskip}
	
	\begin{defn}
	    For dynamics given by \eqref{eq:dynamics}, and an instantaneous cost function $\Phi_i(X(t))$, the best reply strategy is given by
	    
	    \begin{equation}
	        \bar{u}_i = - \nabla_{x_i} \Phi_i(X(t)) \, .
	    \end{equation}
	\end{defn}
        
    \subsection{Two methods for BRS dynamics of the N-player game}\label{sec:MPCdyn}
    
        We now use the MPC approach, as explained in \cite{Albi2015}, \cite{Degond2017} and \cite{Mayne1990}, to simplify the stochastic dynamic game model. This is done through two methods, the first discretises \eqref{eq:obfun} to find an approximate optimal control, while the second discretises the Hamilton-Jacobi-Bellman (HJB) equation related to \eqref{eq:obfun}. As will be seen at the end of this subsection, both methods yield the same BRS control and hence the choice is an arbitrary one to make. Under the MPC approach, we start by assuming that agents only control their behaviour in a piecewise constant manner i.e. $u_i = \sum_{l = 1}^n u_i^l \1_{[(l-i) \Delta t, l \Delta t)}(t)$, where $u_i^l \in \R^d$, and $\Delta t = \frac{T}{n}$, for some $n \in \mathbb{N}$. This reflects the idea that agents do not have continuous control over their dynamics, rather they choose a control for a short time horizon, their dynamics evolve and then they choose a new control. This leads us to

      	\begin{multline} \label{eq:optimisation_approx}
          	\E \left[ \int_0^T \left( \frac{\alpha_i(s)}{2} |u_i(s)|^2 + h_i^{(N)}(X(s)) \right) ds + g_i^{(N)}(X(T)) \right] \\
            = \sum_{l = 1}^n \E \left[ \int_{(l - 1) \Delta t}^{l \Delta t} \left( \frac{\alpha_i(s)}{2} |u_i|^2 + h_i^{(N)}(X(s)) \right) ds + \frac{\Delta t}{T} g_i^{(N)}(X(T)) \right] 
         \end{multline}
         
         We then assume that rather than optimising over the whole time period, each expectation inside the sum in \eqref{eq:optimisation_approx} is optimised at time $l \Delta t$, so
         
         \begin{multline}
         	\E \left[ \int_{(l - 1) \Delta t}^{l \Delta t} \left( \frac{\alpha_i(s)}{2} |u_i|^2 + h_i^{(N)}(X(s)) \right) ds + \frac{\Delta t}{T} g_i^{(N)}(X(T)) \right] \\
            \approx \E \left[ \int_{(l - 1) \Delta t}^{l \Delta t} \left( \frac{\alpha_i(s)}{2} |u_i|^2 + h_i^{(N)}(X(s)) \right) ds + \frac{\Delta t}{T} g_i^{(N)}(X(l \Delta t)) \right].
         \end{multline}

      	The second approximation is correct up to order $O(\Delta t)$.
      	
      	\begin{rmk}
            We are interested in the approximation of controls over each small time period of size $\Delta t$, hence why it is important that in each term the approximation is correct up to order $O(\Delta t)$. From a modelling perspective this would be appropriate in situations where the anticipation of agents is low relative to the length of the time horizon. Of course this will result in a sub-optimal control as the sum of the errors results in an error of order $O(1)$, therefore we cannot necessarily expect the resulting BRS control to approximate the Nash equilibrium control.
      	\end{rmk}
        
        As a result, we may restrict ourselves to the case of considering the dynamics \eqref{eq:dynamics} and control problem \eqref{eq:opt} on the time horizon $[t,t + \Delta t]$. In order for the cost \eqref{eq:obfun} to make sense over such short time horizons, we scale it by $\frac{1}{\Delta t}$. Therefore, under our paradigm, agents are optimising the following expectation over $u_i:\Omega \to \R^d$, where $\Omega$ is the underlying probability space of the SDE \eqref{eq:dynamics}.

        \begin{equation}
            V_i^{\Delta t} (t,x;u) = \E \left[  \int_t^{t + \Delta t} \left( \frac{\alpha_i(s)}{2 \Delta t} |u_i|^2 +\frac{1}{\Delta t} h_i^{(N)}(X(s)) \right) ds + \frac{1}{T} g_i^{(N)}(X(t + \Delta t)) \right]. \label{eq:tempMPCcost}
        \end{equation}

        Where $X(t) = x$ and players are using controls $u$.

    \subsubsection*{Method 1:}
    
    	Using a Riemann sum, specifically the end point quadrature rule, to approximate the integral \eqref{eq:tempMPCcost} up to order $O(\Delta t)$, we get
    
    	\begin{equation}
    		u_i^* = \argmin_{u_i: \Omega \to \R^d} \E  \left[ h_i^{(N)}(X(t + \Delta t)) + \frac{\alpha_i(t + \Delta t)}{2} |u_i|^2 + \frac{1}{T} g_i^{(N)}(X(t + \Delta t)) \right] \, , \qquad i = 1, \ldots, N.
    	\end{equation}
    
    	Using It\^o's formula for $h_i^{(N)}$ at time $t$ and the notation $D_{x_j} := (\partial_{x_{j,k}})_k$, $D_{x_j}^2 := (\partial_{x_{j,k}x_{j,l}})_{k,l}$ gives
    
    	\begin{multline}
    		d h_i^{(N)}(X(t)) = \\
            = \sum_{j = 1}^N \left( (f_j^{(N)}(X(t)) + u_j(t)) \cdot D_{x_j} h_i^{(N)}(X(t)) + \frac{1}{2} \Tr \left( \sigma^2(t,X_j(t)) D_{x_j}^2 h_i^{(N)}(X(t)) \right) \right) dt \\
        	+ \sum_{j = 1}^N \sigma(t,(X_j(t)) D_{x_j} h_i^{(N)}(X(t)) \cdot dB_j(t),
    	\end{multline}
    
    	and similarly for the It\^o expansion of $g_i^{(N)}(X(t + \Delta t))$. We then use an Euler-Maruyama discretisation, a simple extension of the Euler discretisation of an ODE to the setting of SDEs (for more information see \cite{Kloeden1992}), of the dynamics \eqref{eq:dynamics} on $(t,t + \Delta t)$ with an initial value $\bar{X} = X(t)$. We take $\Delta B_i(t) = B_i (t + \Delta t) - B_i(t)$, then take the expectation to get the following weak order $O(\Delta t)$ approximation of $u_i^*$.
    
    	\begin{equation}
        	\begin{aligned}
        		u_i^* =& \argmin_{u_i: \Omega \to \R^d} \E \left[ (h_i^{(N)} + \frac{1}{T} g_i^{(N)})(\bar{X}) + \sum_{j = 1}^N (f_j^{(N)}(\bar{X}) \cdot D_{x_j} (h_i^{(N)} + \frac{1}{T} g_i^{(N)})(\bar{X}) \Delta t \right]  \\
            	&+ \E \left[ \sum_{j = 1}^N \Tr \left( \sigma^2(t, \bar{X}_j) D_{x_j}^2 (h_i^{(N)} + \frac{1}{T} g_i^{(N)})(\bar{X}) \right) \Delta t \right] \\ 
                &+ \E \left[ \sum_{j = 1, \, j \neq i}^N u_j(t) \cdot D_{x_j} (h_i^{(N)} + \frac{1}{T} g_i^{(N)})(\bar{X}) \Delta t \right] \\
            	&+ \E \left[ u_i \cdot D_{x_i} (h_i^{(N)} + \frac{1}{T} g_i^{(N)})(\bar{X}) \Delta t + \frac{\alpha_i(t + \Delta t)}{2} |u_i|^2 \right].
        	\end{aligned}
        	\label{eq:discrete_dyn}
    	\end{equation}
    
    	Notice that only the final term in \eqref{eq:discrete_dyn} depends on $u_i$, so this can be simplified to
    
    	\begin{equation}
    		u_i^* = \argmin_{u_i: \Omega \to \R^d} \E \left[ \left( u_i \cdot D_{x_i} (h_i^{(N)} + \frac{1}{T} g_i^{(N)})(\bar{X}) + \frac{\alpha_i(t + \Delta t)}{2 \Delta t} |u_i|^2 \right) \right]. \label{eq:simplified_dyn}
    	\end{equation}
    
    	Note that, in general, for some functional $ \mathcal{F} = \mathcal{F}(\omega, u(\omega)): \Omega \times \R^d \to \R$ we have 
    
    	\begin{equation}
    		\min_{u(\omega): \Omega \to \R^d} \E [\mathcal{F}(\omega, u(\omega))] = \min_{u(\omega)}\int_{\Omega} \mathcal{F}(\omega, u(\omega)) dP(\omega).
    	\end{equation}
    
    	Now, suppose the function $u^*(\omega)$ satisfies $\mathcal{F}(\omega, u^*(\omega)) = min_{u \in \R} \mathcal{F}(\omega, u)$ for all $\omega$. Then for any other process $u(\omega)$, we necessarily have 
    
    	\begin{equation}
    		\mathcal{F}(\omega, u^*(\omega)) \leq \mathcal{F}(\omega, u(\omega)).
    	\end{equation}
    
    	Integrating over $\Omega$ and taking the minimum with respect to $u(\omega)$ gives
    
    	\begin{equation}
    		\min_{u: \Omega \to \R} \E [\mathcal{F}(\omega, u(\omega))] = \E[\mathcal{F}(\omega, u^*(\omega))]
    	\end{equation}
    
    	Thus, applying this reasoning to \eqref{eq:simplified_dyn}, it is clear that the expectation in \eqref{eq:simplified_dyn} will be minimised if for every $\omega \in \Omega$ the following expression is minimised
    
    	\begin{equation}
    	    \left( u_i(\omega) \cdot D_{x_i} (h_i^{(N)} + \frac{1}{T} g_i^{(N)})(\bar{X}) + \frac{\alpha_i(t + \Delta t)}{2 \Delta t} |u_i(\omega)|^2 \right)
    	\end{equation}
    
    	Now, we fix $\omega \in \Omega$ so that $u_i^* = u_i^*(\omega) \in \R$ is some constant in $\R$ to be found. We can use first order conditions to find $u_i^*$.  Approximating $\alpha_i(t + \Delta t)$ by a Taylor expansion of order $O(\Delta t)$ we get 
    
   		\begin{equation}
    		u_i^* = - \frac{\Delta t}{\alpha_i(t) + \Delta t \dot{\alpha}_i(t)} D_{x_i} (h_i^{(N)} + \frac{1}{T} g_i^{(N)})(\bar{X}).
    	\end{equation}
    	
    	If we were to take the limit $\Delta t \to 0$, we would get $u_i^* = - \frac{1}{\dot{\alpha}_i(t)} D_{x_i} (h_i^{(N)} + \frac{1}{T} g_i^{(N)})(X(t))$, however in many situations we may take $\alpha_i$ as constant, so this control would make no sense. To rectify this problem we have to rescale $\alpha_i$ by $\Delta t$, redefining $V_i^{\Delta t}$ as
    	
    	\begin{equation}
    	    V_i^{\Delta t}(t,x;u) = \E \left[  \int_t^{t + \Delta t} \left( \frac{\alpha_i(s)}{2} |u_i|^2 +\frac{1}{\Delta t} h_i^{(N)}(X(s)) \right) ds + \frac{1}{T} g_i^{(N)}(X(t + \Delta t)) \right]. \label{eq:MPCcost}
    	\end{equation}
    	
    	If we do this and go through the same process as above we get 
    	
    	\begin{equation}
    		u_i^* = - \frac{1}{\alpha_i(t) + \Delta t \dot{\alpha}_i(t)} D_{x_i} (h_i^{(N)} + \frac{1}{T} g_i^{(N)})(\bar{X}).
        	\label{eq:discrete_opt}
    	\end{equation}
    
    	This is known as the BRS, as described in \cite{Degond2014,Degond2014a,Degond2014b} and given by Definition \ref{def:BRS}. In simulations, the dynamics given by substituting \eqref{eq:discrete_opt} into the discretised version of \eqref{eq:dynamics} would be calculated and the new state $X(t + \Delta t)$ would be used to repeat the process. If we now let $\Delta t \to 0$, we get the following dynamics
    
    	\begin{equation}
    		\begin{aligned}
    			&dX_i(t) = \left( f_i^{(N)}(X(t)) - \frac{1}{\alpha_i(t)} D_{x_i} (h_i^{(N)} + \frac{1}{T} g_i^{(N)})(X(t)) \right) dt + \sigma(t, X_i(t)) d B_i(t) \label{eq:mpcdyn} \\
            	&X_i(0) = X_{i,0}.
        	\end{aligned}
    	\end{equation}

    \subsubsection*{Method 2:}
    
    	This method uses the well know fact (see \cite{Friedman1972} or \cite{Oksendal2007} for example) that $V_i(t,x)$ solves the following HJB equation for every $i = 1, \ldots, N$.
    
    	\begin{equation}
    		\begin{aligned}
    			\sup_{u_i \in \R^d} \bigg\{ \frac{\alpha_i(t)}{2} |u_i|^2 + h_i^{(N)}(x) + \partial_t V_i(t,x) +  \\
            	\sum_{j = 1, \, j \neq i}^N \left(f_j^{(N)}(x) + u_j^*\right) \cdot D_{x_j} V_i(t,x) + \left(f_i^{(N)}(x) + u_i\right) \cdot D_{x_i} V_i(t,x)  \\
            	+ \frac{1}{2} \sum_{j = 1}^N \Tr \left( \sigma^2(t,x_j) D_{x_j}^2 V_i(t,x) \right) \bigg\} = 0, \label{eq:HJB}
    		\end{aligned}
    	\end{equation}
    
    	with $V_i(T,x) = g_i^{(N)}(x) \, \forall x \in \R^{d \times N}$. Using first order conditions, we find that for every $i = 1, \ldots, N$
    
    	\begin{equation}
    		u_i^* = - \frac{1}{\alpha_i(t)} D_{x_i} V_i(t,x). \label{eq:optcontrol}
    	\end{equation}
    
    	Substituting the above into \eqref{eq:HJB}, we obtain the following HJB equation and optimal agent dynamics:
    
    	\begin{equation}
    		\begin{aligned}
    			\frac{1}{2\alpha_i(t)} |D_{x_i} V_i(t,x)|^2 &= h_i^{(N)}(x) + \partial_t V_i(t,x) + f_i^{(N)}(x) \cdot D_{x_i} V_i(t,x) \\ 
            	&+ \sum_{j = 1, \, j \neq i}^N \left(f_j^{(N)}(x) - \frac{1}{\alpha_j(t)} D_{x_j} V_j(t,x)\right) \cdot D_{x_j} V_i(t,x) \\ 
            	&+\frac{1}{2} \sum_{j = 1}^N \Tr \left( \sigma^2(t,x_j) D_{x_j}^2 V_i(t,x) \right)          
            	\label{eq:HJB2}
    		\end{aligned}
		\end{equation}
    
    	\begin{equation}        
        	dX_i(t) = \left( f_i^{(N)}(X(t) - \frac{1}{\alpha_i(t)} D_{x_i} V_i(t,X(t)) \right) dt + \sigma(t,X_i(t)) dB_i(t).
    	\end{equation}
    
    	Note that we still have the same terminal condition for the PDE for $V_i$ and initial condition for the SDE for $X_i$.
    
        Now, using the MPC approach, we actually want to consider $V_i^{\Delta t}$ rather than $V_i$. We discretise the analogue of (2.22) for $V_i^{\Delta t}$ (which is found by replacing $h_i^{(N)}$ by $\frac{h_i^{(N)}}{\Delta t}$) in the time direction only to get an order $O(\Delta t )$ approximation of $V_i^{\Delta t}$. This is a backward in time discretisation since we are given a terminal condition. This results in the following equation
    
    	\begin{equation}
    		\begin{aligned}
    			\frac{1}{2\alpha_i(t + \Delta t)} |D_{x_i} V_i^{\Delta t}(t + \Delta t,x)|^2 = \frac{h_i^{(N)}(x)}{\Delta t} + \frac{V_i^{\Delta t}(t + \Delta t,x) - V_i^{\Delta t}(t,x)}{\Delta t} \\
                + f_i^{(N)}(x) \cdot D_{x_i} V_i^{\Delta t}(t + \Delta t,x) \\ 
                + \sum_{j = 1, \, j \neq i}^N \left(f_j^{(N)}(x) - \frac{1}{\alpha_j(t)} D_{x_j} V_j^{\Delta t}(t + \Delta t,x)\right) \cdot D_{x_j} V_i^{\Delta t}(t + \Delta t,x) \\
            	+ \frac{1}{2} \sum_{j = 1}^N  \Tr \left( \sigma^2(t + \Delta t,x_j) D_{x_j}^2 V_i^{\Delta t}(t + \Delta t,x) \right) .
    		\end{aligned}    	
    	\end{equation}
    
    	Since we have $V_i^{\Delta t}(t + \Delta t,x) = \frac{1}{T} g_i^{(N)}(x)$ for all $x \in \R^N$, this yields an order $O(\Delta t)$ approximation $V_i^{\Delta t}(t,x) = (h_i^{(N)} + \frac{1}{T} g_i^{(N)})(x)$, which returns us to \eqref{eq:discrete_opt} and \eqref{eq:mpcdyn}. Thus we can conclude that the BRS can be derived from an MPC approach for stochastic differential games either through the discretisation of the value function or through the discretisation of the corresponding HJB equation.
    
    \subsection{Deriving the BRS dynamics for the limit $N \to \infty$} \label{sec:MPClim}
    
    	We now look at the limiting behaviour of equation \eqref{eq:mpcdyn} as $N \to \infty$. First, we shall make the following assumptions on the symmetry of  $f_i^{(N)}$, $h_i^{(N)}$ and $\alpha_i$, in order to pass to the limit as $N \to \infty$ in a coherent manner. Similar such assumptions are made throughout MFG literature, c.f. \cite{Blanchet2014,Cardaliaguet2010,Carmona2012,Carmona2012a,Degond2017,Huang2006b,Feleqi2013ThePlayers}. The assumptions are:
    
    	\begin{itemize}
   			\item [\textbf{A}] $f_i^{(N)}(x) = f(x_i,m_{-i}^N)$ for some $f:\R^d \times \mathcal{P}(\R^d) \to \R^d$, where $\mathcal{P}(\R)$ is the set of Borel probability measures on $\R^d$ and $m_{-i}^N = \frac{1}{N - 1} \sum_{j = 1, \, j \neq i}^N \delta_{x_j}$.
    
			\item [\textbf{B}] Using the same notation, $h_i^{(N)}(x) = h(x_i,m_{-i}^N)$ for some $h:\R^d \times \mathcal{P}(\R^d) \to \R$. As a direct consequence, $D_{x_i} h_i^{(N)}(x) = D_x h(x_i,m_{-i}^N)$, similarly for $g_i^{(N)}$.
        
        	\item[\textbf{C}] In order to ensure we can use symmetry arguments later in this paper, we require $\alpha_i(t) = \alpha(t)$ for all $i = 1, \ldots, N$.
        	
   		\end{itemize}
   
		Using assumptions \textbf{A}, \textbf{B} and \textbf{C}, equation \eqref{eq:mpcdyn} can be rewritten as
		\begin{equation}
			dX_i(t) = \left( f(X_i(t), M_{-i}^N(t)) - \frac{1}{\alpha(t)} D_x (h + \frac{1}{T} g)(X_i(t), M_{-i}^N(t)) \right) dt + \sigma(t, X_i(t)) dB_t^i.
		\end{equation}
   
		Here, $M_{-i}^N(t) = \frac{1}{N - 1} \sum_{j = 1, \, j \neq i}^N \delta_{X_j(t)}$. We also require growth assumptions to apply results relating to the limit of this SDE as $N \to \infty$ as well as for existence and uniqueness later. The assumptions made are
    
    	\begin{itemize}
        	\item [\textbf{D}] For every $i = 1, \ldots, N$ $X_i^0$, as defined in \eqref{eq:dynamics} are i.i.d. random variables with distribution $m_0 \in \mathcal{P}_2(\R^d)$.
        	
    		\item [\textbf{E}] There exists some constants $C_1 > 0$ such that for any $(x_1,m_1),(x_2,m_2) \in \R^d \times \mathcal{P}(\R^d)$, we have $|f(x_1,m_1) - f(x_2,m_2)| \leq C_1(|x_1 - x_2| + \textbf{d}_1(m_1,m_2))$.
        
        	\item [\textbf{F}] There exists some constant $C_2 > 0$ such that for any $(x_1,m_1),(x_2,m_2) \in \R^d \times \mathcal{P}(\R^d)$, we have $|D_x h(x_1,m_1) - D_x h(x_2,m_2)| \leq C_2(|x_1 - x_2| + \textbf{d}_1(m_1,m_2))$, and similarly for $D_x g$.
        
        	\item [\textbf{G}] There exists some constant $C_3 > 0$ such that for any $(t_1,x_1),(t_2,x_2) \in [0,T] \times \R^d$, we have $|\sigma(t_1,x_1) - \sigma(t_2,x_2)| \leq C_3(|t_1 - t_2| + |x_1 - x_2|)$.
    	\end{itemize}
    
    	Here, $\textbf{d}_1(m_1,m_2)$ is the $1$-Wasserstein distance, see Appendix \ref{sec:meas} for further information on the Wasserstein distance and the space of probability measures. Under the assumptions \textbf{D}, \textbf{E}, \textbf{F} and \textbf{G}, trivial modifications to theorem 1.3 in \cite{Jourdain2007} and theorem 1.4 in \cite{Sznitman1991} show that for every $i = 1, \ldots, N$, $X_i^{(N)}(t)$ converges as $N \to \infty$ (uniformly in $t$ and $i$) in any compact interval $[0,T]$ in the $L^2$ sense to the following McKean-Vlasov equation.
   
		\begin{equation}
    		\begin{aligned}
    			& dX_t = \left( f(X_t, m_t) - \frac{1}{\alpha(t)} D_x (h + \frac{1}{T} g)(X_t, m_t)) \right) dt + \sigma(t, X_t) dB(t), \label{eq:limit_dyn} \\
            	& X_{t = 0} = X_0
    		\end{aligned}
		\end{equation}
    
    	where $m_t$ is the law of $X_t$, $B_t$ is a $d$-dimensional Wiener process and $X_0$ is an $m_0$-distributed independent random variable. Such behaviour, of particles in large systems behaving independently of one another, is known as the propagation of chaos. Together, \cite{Sznitman1991} and \cite{Jourdain2007} effectively cover existence and uniqueness of such SDEs. In fact, they have shown that for such SDEs strong existence and uniqueness holds under the assumptions on $f$, $D_x h$, $D_x g$ and $\sigma$ already made. It is often also of interest to understand how the entire population of players develops over time. To do this, we will describe the evolution of the distribution of the population by analysing the weak form of a pde for measures. For more information on differentiability in the space of measures, see Appendix \ref{sec:meas}.
     
    	Let $\phi \in C_c^{\infty}([0,T] \times \R^d)$ be any test function, here $\phi \in C_c^{\infty}([0,T] \times \R^d)$ is the set of compactly supported, infinitely differentiable functions from $[0,T] \times \R^d$ to $\R$. By It\^o's formula we have
     
    	\begin{equation}
    	    \begin{aligned}
     		    \phi(t,X_t) & = \phi(0,X_0) \\
                & + \int_0^t \left[ \partial_s \phi(s,X_s) + D_x \phi(s,X_s) \cdot \left( f(X_s,m_s) - \frac{1}{\alpha(s)} D_x (h + \frac{1}{T} g)(X_s,m_s) \right) \right] ds \\
                & + \frac{1}{2} \int_0^t \Tr \left( \sigma^2(s,X_s) D_x^2 \phi(s,X_s) \right) ds + \sum_{k=1}^d \int_0^t \partial_{x_k}\phi(s,X_s) \sigma_{kk}(s,X_s) dB_k(s).
            \end{aligned}
    	\end{equation}
     
    	We can then take the expectation, and use the boundedness of $\phi$ together with the Lipschitz properties of the other functions to bring the expectation inside the integral via the dominated convergence theorem, so we have
     
    	\begin{equation}
    	    \begin{aligned}
			    \E[\phi(t,X_t)] & - \E[\phi(0,X_0)] - \int_0^t \E[\partial_s \phi(s,X_s)] = \\
			    & = \int_0^t \E  \left[ D_x \phi(s,X_s) \cdot \left( f(X_s,m_s) - \frac{1}{\alpha(s)} D_x (h + \frac{1}{T} g)(X_s,m_s) \right) \right] ds \\
			    & + \int_0^t \E \left[ \frac{1}{2} \Tr \left( \sigma^2(s,X_s) D_x^2 \phi(s,X_s) \right) \right] ds.		     	
    	    \end{aligned}
    	\end{equation}
    
    	Finally, we can use the definition of $m_s$ as the law of $X_s$ and evaluate the expectations. This gives the weak form of the following PDE that $m_t$ must satisfy
    
    	\begin{equation}
    		\begin{aligned}
    			& \partial_t m_t = D_x \cdot \left[ \left( \frac{1}{\alpha(t)} D_x (h + \frac{1}{T} g)(x,m_t) - f(x,m_t) \right) m_t \right] + \frac{1}{2} \Tr \left( D_x^2 (\sigma^2(t,x) m_t) \right) \\
            	& m_{t= 0} = m_0 \label{eq:measurePDE}
    		\end{aligned}
		\end{equation}
    
\section{Relating the BRS to MFGs} \label{sec:MFGMPC}
    
    The aim of this section is to demonstrate how the BRS is linked to MFGs and to begin demonstrating the extent of the simplification made in the BRS
    
    \subsection{Mean field limit of the HJB equation}
    
    	We will now analyse the limiting behaviour of the HJB equation \eqref{eq:HJB2}. To this end we require that assumptions \textbf{A}, \textbf{B},\textbf{C}, \textbf{D}, \textbf{E} and \textbf{F} all hold. In order to describe the limiting case, we need to consider a function $W = W(t,x,Y): [0,T] \times \R^d \times \R^{d(N - 1)} \to \R$, as in \cite{Degond2017}, which we may relate to $V_i$.
        
        \vspace{\parskip}
        
        \begin{lem} \label{lemma:mfg_limit_HJB}
        	Assume there exists a function $W = W(t,x,Y): [0,T] \times \R^d \times \R^{d(N - 1)} \to \R$ that is $C^1$ in $t$, $C^2$ in $x$ and $C^2$ in $Y$ and which satisfies 
            
            \begin{dgroup}
    		\begin{dmath}
    			\frac{1}{2 \alpha(t)} |D_x W(t,x,Y)|^2 = h(x,m_Y) + \partial_t W(t,x,Y) + f(x,m_Y) D_x W(t,x,Y) + \sum_{j = 1}^{N-1} \left(f(y_j,m_{Y_j^x}) - \frac{1}{\alpha(t)} D_{x} W(t,y_j,Y_j^x)\right) \cdot D_{y_j} W(t,x,Y) + \frac{1}{2} \sum_{j = 1}^{N-1} \Tr \left( \sigma^2(t,y_j) D_{y_j}^2 W(t,x,Y) \right) + \Tr \left( \sigma^2(t,x) D_x^2 W(t,x,Y) \right)
    		\end{dmath}
    
    		\begin{dmath}
    			W(T,x,Y) = g(x,m_Y),
    		\end{dmath}
    		\label{eq:symHJB}
    	\end{dgroup}
        
        where $m_Y = \frac{1}{N-1}\sum_{j=1}^{N-1} \delta_{y_j}$ and $Y_k^x = (y_1, \ldots,y_{k-1},x,y_{k+1},\ldots,y_{N-1})$. Then $V_i$, satisfying \eqref{eq:obfun} has a mean field limit $\textbf{W} (t,x,m): [0,T] \times \R^d \times \mathcal{P}(\R^d) \to \R$. That is, denoting $V_i$ by $V_i^N$ to emphasise the dependence on $N$, if $m_X \overset{*}{\rightharpoonup} m$ as $N \to \infty$ then $V_i^N(t,X) \to \textbf{W}(t,x_i,m)$ as $N \to \infty$. Furthermore, $\textbf{W}$ satisfies the following PDE in the weak sense
        
        \begin{equation}
    		\begin{aligned}
    			\frac{1}{2 \alpha(t)} |D_x \textbf{W}(t,x,m)|^2 =& h(x,m) + \partial_t \textbf{W}(t,x,m) \\ 
                &+ f(x,m) \cdot D_x \textbf{W}(t,x,m) + \frac{1}{2} \Tr \left( \sigma^2(t,x) D_x^2 \textbf{W}(t,x,m) \right)  \\
    			&+ \int_{\R^d} \left( f(y,m) - \frac{1}{\alpha(t)} D_x \textbf{W}(t,y,m) \right) \cdot \partial_m \textbf{W}(t,x,m)(y) m(dy) \\
    			&+ \frac{1}{2} \int_{\R^d} \Tr \left(  \sigma^2(t,y) D_y \left[\partial_{m} \textbf{W}(t,x,m)\right](y) \right) m(dy), \\ 
    	        \textbf{W}(T,x,m) = g(x). \qquad& 
    		\end{aligned} \label{eq:mfhjb}
    	\end{equation}        
        \end{lem}
        
        See Appendix \ref{sec:meas} for more information on differentiation with respect to measures.
        
        \begin{proof}
        Let's assume $W$ satisfies \eqref{eq:symHJB}, then we can define
        
        \begin{equation}
    		V_i(t,X) = W(t,x_i,X_{-i}),
    	\end{equation}
        
        where $X \in \R^{dN}, \, X_{-i} \in \R^{d(N-1)}$. Then the partial derivatives of $V_i$ satisfy
    
    	\begin{align*}
    		D_{x_i} V_i(t,X) &= D_x W(t,x_i,X_{-i}) \\
        	\partial_t V_i(t,X) &= \partial_t W(t,x_i,X_{-i}) \\
        	D_{x_j} V_j(t,X) &= D_x W(t,x_j,X_{-j}) \\
        	D_{x_j} V_i(t,X) &= D_{y_j} W(t,x_i,X_{-i}) \\
        	D_{x_j}^2 V_i(t,X) &= D_{y_j}^2 W(t,x_i,X_{-i}) \\
        	D_{x_i}^2 V_i(t,X) &= D_x^2 W(t,x_i,X_{-i}).
    	\end{align*}
    
    	Using the above identities along with assumptions \textbf{A} and \textbf{B}, it is clear that $W(t,x_i,X_{-i})$ is a solution to \eqref{eq:HJB2}. Therefore, we no longer have to concern ourselves with studying the $N$ equations of \eqref{eq:HJB2} for $i = 1, \ldots, N$. Instead we can look at the behaviour of \eqref{eq:symHJB}, as in \cite{Degond2017}, in particular we can look at the limiting case of \eqref{eq:symHJB} as $N \to \infty$. Note that since $f$, $g$ and $h$ only depend on the empirical distribution of $Y$, the solution, $W$, to \eqref{eq:symHJB} will also only depend on the empirical distribution of $Y$. Thus, there exists a function $\textbf{W}:[0,T] \times \R^d \times \mathcal{P}(\R^d) \to \R$ such that 
    
    	\[ W(t,x,Y) = \textbf{W}(t,x,m_Y) \]
    
    	The partial derivatives of $W$, and hence $V_i$, can also be seen as partial derivatives of $\textbf{W}$. We have
    
    	\begin{align*}
    		\frac{1}{2 \alpha(t)}|D_x W(t,x,Y)|^2 &= \frac{1}{2 \alpha(t)}|D_x \textbf{W}(t,x,m_Y)|^2 \\
        	\partial_t W(t,x,Y) &= \partial_t \textbf{W}(t,x,m_Y) \\
        	\sigma^2(t,x) D_x^2 W(t,x,Y) &= \sigma^2(t,x) D_x^2 \textbf{W}(t,x,m_Y). 
    	\end{align*}
    
    	We now have two remaining terms to consider, namely
    
    	\begin{align}
    		& \mathcal{C}_1 = \sum_{k = 1}^{N-1} \left(f(y_k,m_{Y_k^x}) - \frac{1}{\alpha(t)} D_x W(t,y_k,Y_k^x)\right) \cdot D_{y_k} W(t,x,Y)  \label{eq:sum1}\\
        	& \mathcal{C}_2 = \frac{1}{2} \sum_{j = 1}^{N-1} \Tr \left( \sigma^2(t,y_j) D_{y_j}^2 W(t,x,Y) \right) \label{eq:sum2}
    	\end{align}
    
    	It is possible to modify Proposition 6.1 in \cite{Cardaliaguet2015a} in order to rigorously express each of the terms through derivatives of the mean field limit $\textbf{W}$. Using this proposition, the following identities hold:
    
    	\begin{align}
    		D_x W(t,y_k,Y_k^x) &= D_x  \textbf{W}(t,y_k,m_{Y_k^x}) \\
        	D_{y_k} W(t,x,Y) &= \frac{1}{N-1} \partial_m \textbf{W}(t,x,m_Y)(y_k) \\
        	D_{y_j}^2 W(t,x,Y) &= \frac{1}{(N-1)^2} \partial_m^2 \textbf{W}(t,x,m_Y)(y_j,y_j) + \frac{1}{N-1} D_y \left[ \partial_m\textbf{W}(t,x,m_Y) \right] (y_j)
    	\end{align}
        
        In the above, $\partial_m^2 \textbf{W} := D_y \frac{\delta}{\delta m}(\partial_m\textbf{W})$, see \cite{Cardaliaguet2015a} and Appendix \ref{sec:meas} for further information on differentiability in the space of measures. Placing these identities into \eqref{eq:sum1} -- \eqref{eq:sum2} and simplifying using the definition of $m_Y$ we get
    
    	\begin{align}
    		\mathcal{C}_1 = \frac{1}{N-1} \sum_{k = 1}^{N-1} \left(f(y_k,m_{Y_k^x}) - \frac{1}{\alpha(t)} D_x \textbf{W}(t,y_k,m_{Y_k^x})\right) \cdot \partial_m \textbf{W}(t,x,m_Y)(y_k)  \label{eq:mf1}\\
            \begin{aligned}
            	\mathcal{C}_2 = \frac{1}{2(N-1)} \int_{\R^d} \Tr \left( \sigma^2(t,y) \partial_m^2 \textbf{W}(t,x,m_Y)(y) \right) m_Y(dy) \\
                + \frac{1}{2} \int_{\R^d} \Tr \left( \sigma^2(t,y) D_y \partial_{m} \textbf{W}(t,x,m_Y)(y) \right) m_Y(dy) 
            \end{aligned}            
        	\label{eq:mf2}
    	\end{align}
        
        We are now in a position to take the limit $N \to \infty$. First we assume that $m_Y \overset{*}{\rightharpoonup} m$ as $N \to \infty$. In this case, it is a matter of computation to show, for any $\phi \in C_c^{\infty}(\R^d)$, there exists an $\epsilon > 0$ such that
    
    	\begin{equation}
    		\left\vert \int_{\R} \phi(z) (m_{Y_k^x} - m) (dz) \right\vert \leq \left\vert \int_{\R} \phi(z) (m_Y - m) (dz) \right\vert + \frac{1}{N-1} \left\vert \phi(x) - \phi(y_k) \right\vert < \epsilon.
    	\end{equation}
    
    	Hence, $m_{Y_k^x} \overset{*}{\rightharpoonup} m$ as $N \to \infty$ as well. So, finally we find the mean field equation for $\textbf{W}(t,x,m)$ to be
    
		\begin{equation}
    		\begin{aligned}
    			\frac{1}{2 \alpha(t)} |D_x \textbf{W}(t,x,m)|^2 =& h(x,m) + \partial_t \textbf{W}(t,x,m) \\ 
                &+ f(x,m) \cdot D_x \textbf{W}(t,x,m) + \frac{1}{2} \Tr \left( \sigma^2(t,x) D_x^2 \textbf{W}(t,x,m) \right)  \\
    			&+ \int_{\R^d} \left( f(y,m) - \frac{1}{\alpha(t)} D_x \textbf{W}(t,y,m) \right) \cdot \partial_m \textbf{W}(t,x,m)(y) m(dy) \\
    			&+ \frac{1}{2} \int_{\R^d} \Tr \left( \sigma^2(t,y) D_y \left[\partial_{m} \textbf{W}(t,x,m)\right](y) \right) m(dy), \\ 
    	        \textbf{W}(T,x,m) = g(x). \qquad& 
    		\end{aligned}
    	\end{equation}
        
        \end{proof}
    
	\subsection{Mean field limit of the player dynamics}
    
    	Now that we have seen the mean field limit of the HJB equation, this needs to be married to the mean field limit of the player dynamics \eqref{eq:dynamics} with the optimal control, given by \eqref{eq:optcontrol}. 
        
        \vspace{\parskip}
        
        \begin{lem}
        Under assumptions \textbf{A}--\textbf{F} and the assumptions of Lemma \ref{lemma:mfg_limit_HJB}, then for every $i \in \{1, \ldots, N\}$, $X_i$ satisfying \eqref{eq:dynamics} converges to a random process satisfying the following McKean-Vlasov equation as $N \to \infty$ (uniformly in $t$ and $i$) in any compact interval $[0,T]$ in the $L^2$ sense.
        
        \begin{equation}
    		\begin{aligned}
    			&dX_t = (f(X_t,m_t) - \frac{1}{\alpha(t)} D_x \textbf{W}(t,X_t,m_t))dt + \sigma(t,X_t)dB_t. \\
    	    	&X_{t=0} = X_0 \label{eq:fullMFGdyn}
    	    \end{aligned}
    	\end{equation}
        
        Here, $B_t$ is a $d$-dimensional Wiener process, $m_t$ is the law of $X_t$ and $X_0$ is an independent random variable with law $\mathcal{L}(X_0) = m_0$, where $m_0$ is the limit as the number of players goes to infinity of the law of $X_{i,0}$ in the finite-agent case.
        
        \end{lem}
        
        \vspace{\parskip}
    
        \begin{rmk}
        	In a similar way to \eqref{eq:measurePDE}, it is possible to go  from \eqref{eq:fullMFGdyn} to the following continuity equation for $m_t$:
    
    	\begin{equation}
    		\begin{aligned}
    			& \partial_t m_t = D_x \cdot \left[ \left( \frac{1}{\alpha(t)} D_x \textbf{W}(t,x,m_t) - f(x,m_t) \right) m_t \right] + \frac{1}{2} \Tr \left( D_x^2 (\sigma^2(t,x) m_t) \right) \\
    	        & m_{t= 0} = m_0. 
    		\end{aligned} \label{eq:mfpd}
    	\end{equation}
        \end{rmk}
        
        \vspace{\parskip}
        
        \begin{proof}
        	 In light of the assumptions in Lemma \ref{lemma:mfg_limit_HJB}, and following the previous section's definition of $\textbf{W}(t,x,m)$, the dynamics of the N-player game can be reformulated as
    
    	\begin{equation}
    		\begin{aligned}
    			&dX_i(t) = (f(X_i(t),m_{-i}^N(t)) - \frac{1}{\alpha(t)} D_x \textbf{W}(t,X_i(t),m_{-i}^N(t)))dt + \sigma(t,X_i(t))dB_t^i \\
    	    	&X_i(0) = X_{i,0}.
    		\end{aligned}
    	\end{equation}
    
    	Here, $m_{-i}^N(t) = \sum_{j=1, \,j \neq i}^N \delta_{X_j(t)}$ and for every $i = 1, \ldots,N$, $X_{i,0}$ is an i.i.d. random variable with disribution $m_0$. Using the propagation of chaos, as in Section \ref{sec:MPC}, and assumptions \textbf{A}--\textbf{F}, we can conclude the result.
        \end{proof}

	\subsection{The mean field equation} \label{sec:mfe}
    
    Now that we have, by \eqref{eq:mfpd} and \eqref{eq:mfhjb}, the mean field limits of the player dynamics and HJB equation, we want to find how the HJB equation changes along characteristics governed by the mean field player dynamics. The following theorem describes that.
    
    \vspace{\parskip}
    
    \begin{thm} \label{thm:mean_limit}
    	Let's define 
        
        \begin{equation}
        	\textbf{w}(t,x) := \textbf{W}(t,x,m_t),
        \end{equation}
        
        where $m_t$ is satisfies \eqref{eq:mfpd}. Then $\textbf{w}(t,x)$ satisfies the following mean field PDE
        \begin{equation} \label{eq:mfg_value_fn}
        	\begin{aligned}
				&\frac{1}{2 \alpha(t)} |D_x \textbf{w}(t,x)|^2 = h(x,m_t) + \partial_t \textbf{w}(t,x) + f(x,m_t) \cdot D_x \textbf{w}(t,x) + \frac{1}{2} \Tr \left( \sigma^2(t,x) D_x^2 \textbf{w}(t,x) \right), \\
		    	&\textbf{w}(T,x) = g(x). 
			\end{aligned}
        \end{equation}

        This system together with \eqref{eq:mfpd} then fully describe the mean field limit of the dynamics of the $N$-player stochastic differential game
    \end{thm}
    
    \vspace{\parskip}
    
    \begin{rmk}
    	With this theorem, we have also described the system that corresponds to $\epsilon$-Nash equilibria of the $N$-player game \eqref{eq:dynamics} (see e.g. \cite{Cardaliaguet2010} for how solutions to the mean field game correspond to Nash equilibria of the finite-player game).
    \end{rmk}

	\vspace{\parskip}
    
    \begin{proof}
    	Our first step is to describe $\partial_t \textbf{w}(t,x)$. Note that

		\begin{equation}
			\partial_t \textbf{w}(t,x) = \frac{d}{dt} \textbf{W}(t,x,m_t) = \lim_{h \to 0} \frac{\textbf{W}(t+h,x,m_{t+h}) - \textbf{W}(t,x,m_t)}{h}.
		\end{equation}

		Adding and subtracting $\textbf{W}(t+h,x,m_t)$, we get

		\begin{equation}
			\partial_t \textbf{w}(t,x) = \partial_t \textbf{W}(t,x,m_t) + \lim_{h \to 0} \frac{\textbf{W}(t+h,x,m_{t+h}) - \textbf{W}(t+h,x,m_t)}{h}.
		\end{equation}

		Following the method of (25) in \cite{Cardaliaguet2015a}, we get

		\begin{equation}
			\textbf{W}(t+h,x,m_{t+h}) - \textbf{W}(t+h,x,m_t) = \int_0^1 \int_{\R^d} \frac{\delta \textbf{W}}{\delta m}(t+h,x,(1-s)m_t + sm_{t+h})(y) (m_{t+h} - m_t)(dy) ds. \label{eq:mfeW}
		\end{equation}

		Using \eqref{eq:mfpd}, with a test function $\phi(y) := \frac{\delta \textbf{W}}{\delta m}(t+h,x,(1-s)m_t + sm_{t+h})(y)$, we see \eqref{eq:mfeW} leads to

		\begin{equation}
			\begin{aligned}
            &\textbf{W}(t+h,x,m_{t+h}) - \textbf{W}(t+h,x,m_t) = \\
				&= \int_0^1 \int_t^{t+h} \int_{\R^d} \frac{1}{2} \Tr( \left( \sigma^2(\tau,y) D_y^2 \frac{\delta \textbf{W}}{\delta m}(t+h,x,(1-s)m_t + sm_{t+h})(y) \right)  m_{\tau}(dy) d\tau ds \\
		        &+ \int_0^1 \int_t^{t+h} \int_{\R^d} \left( f(y,m_{\tau}) - \frac{1}{\alpha(\tau)} D_x \textbf{W}(\tau,y,m_{\tau}) \right) \cdot \\
                & \qquad \qquad \qquad \qquad \quad \qquad \cdot \left( D_y \frac{\delta \textbf{W}}{\delta m}(t+h,x,(1-s)m_t + sm_{t+h})(y) \right)  m_{\tau}(dy) d\tau ds
			\end{aligned}           
		\end{equation}

		We now divide by $h$, take the limit $h \to 0$ and, noting $\partial_m \textbf{W} :=  D_y \frac{\delta \textbf{W}}{\delta m}$, we get

		\begin{equation}
			\begin{aligned}
				\partial_t \textbf{w}(t,x) = \partial_t \textbf{W}(t,x,m_t) + \int_{\R^d} \left( f(y,m_t) - \frac{1}{\alpha(t)} D_x \textbf{W}(t,y,m_t) \right) \cdot \partial_m \textbf{W}(t,x,m_t)(y) m_t(dy) & \\
		        + \frac{1}{2} \int_{\R^d} \Tr \left( \sigma^2(t,y) D_y [\partial_m \textbf{W}(t,x,m_t)](y) \right) m_t(dy) &.
			\end{aligned}
		\end{equation}

		We can therefore substitute this into \eqref{eq:mfhjb} to obtain \eqref{eq:mfg_value_fn}
    \end{proof}

		So, being able to solve \eqref{eq:mfpd} and \eqref{eq:mfg_value_fn} gives a solution along characteristics for $\textbf{W}$. This corollary summarsises the result
        \vspace{\parskip}
        
        \begin{cor}
        	Provided assumptions \textbf{A} -- \textbf{F}  and the assumptions in Lemma \ref{lemma:mfg_limit_HJB} hold, the mean field limit of the HJB equation governing the stochastic differential game \eqref{eq:dynamics}, along with the evolution in time of the distribution of agents is given by 

		\begin{align}
			&\frac{1}{2 \alpha(t)} |D_x \textbf{w}(t,x)|^2 = h(x,m_t) + \partial_t \textbf{w}(t,x) + f(x,m_t) \cdot D_x \textbf{w}(t,x) + \frac{1}{2} \Tr \left( \sigma^2(t,x) D_x^2 \textbf{w}(t,x) \right), \label{fullmf1} \\
		    & \partial_t m_t = D_x \cdot \left[ \left( \frac{1}{\alpha(t)} D_x \textbf{w}(t,x) - f(x,m_t) \right) m_t \right] + \frac{1}{2} \Tr \left( D_x^2 (\sigma^2(t,x) m_t) \right) \label{fullmf2},
		\end{align}

		with initial and terminal conditions:

		\begin{align}
			\textbf{w}(T,x) &= g(x) \\
		    m_{t= 0} &= m_0.
		\end{align}

		Equation \eqref{fullmf1} describes the evolution (backwards in time) of an individual agent's expected cost, while \eqref{fullmf2} describes the evolution (forwards in time) of the distribution of all agent in the state space.
        \end{cor}

	\subsection{The mean field BRS for the MFG system of equations}

		Having calculated the mean field equations in the previous section, we now use the MPC approach to obtain the mean field BRS from the MFG and find that it matches \eqref{eq:measurePDE}. Clearly, in terms of numerical calculations, it is much simpler to use the best reply strategy dynamics, \eqref{eq:measurePDE} as an approximation to \eqref{fullmf1} and \eqref{fullmf2} rather than try to solve these two equations directly. This section will show that the BRS is indeed a simplification of \eqref{fullmf1} and \eqref{fullmf2} using MPC methods.
        
        \vspace{\parskip}
        
        \begin{thm}
        	Define $V_i^{\Delta t}(t,x)$ by 
        
        	\begin{equation}
        		V_i^{\Delta t}(t,x) = \min_{u_i \in \mathcal{A}} \E \left[  \int_t^{t + \Delta t} \left( \frac{\alpha_i(s)}{2} |u_i|^2 +\frac{1}{\Delta t} h_i^{(N)}(X(s)) \right) ds + \frac{1}{T} g_i^{(N)}(X(t + \Delta t)) \right],
        	\end{equation}
        
            where $X(s)$ solves \eqref{eq:dynamics} with controls 
            
            \begin{equation}
            	u_i^* =  \argmin_{u_i \in \mathcal{A}} \E \left[  \int_t^{t + \Delta t} \left( \frac{\alpha_i(s)}{2} |u_i|^2 +\frac{1}{\Delta t} h_i^{(N)}(X(s)) \right) ds + \frac{1}{T} g_i^{(N)}(X(t + \Delta t)) \right],
            \end{equation}
            
            and $X(t) = x$. Then the HJB equation associated with $V_i^{\Delta t}$ has a mean field limit
            
            \begin{equation}
            	\begin{aligned}
				&\frac{1}{2 \alpha(t)} |D_x \textbf{w}(t,x)|^2 = \frac{h(x,m_t)}{\Delta t} + \partial_t \textbf{w}(t,x) + f(x,m_t) \cdot D_x \textbf{w}(t,x) + \frac{1}{2} \Tr \left( \sigma^2(t,x) D_x^2 \textbf{w}(t,x) \right), \\
		    	&\textbf{w}(t + \Delta t,x) = \frac{1}{T} g(x),
			\end{aligned}
            \end{equation}
            
            the solution of which can be approximated up to an error of order $O(\Delta t)$ by 
            
            \begin{equation}
            	\textbf{w}^{\Delta t}(t,x) = (h + \frac{1}{T} g)(x,m_t).
            \end{equation}
            
            The corresponding law of motion for the distribution of players in the mean field limit is therefore given (up to an error of order $O(\Delta t)$) by
            
            \begin{equation} \label{eq:small_mean_field_limit}
	    	\begin{aligned}
	    		& \partial_t m_t = D_x \cdot \left[ \left( \frac{1}{\alpha(t)} D_x (h + \frac{1}{T} g)(x,m_t) - f(x,m_t) \right) m_t \right] + \frac{1}{2} \Tr \left( D_x^2 (\sigma^2(t,x) m_t) \right) \\
	            & m_{t= 0} = m_0.
	    	\end{aligned}
		\end{equation}
         
        \end{thm}
        
        \begin{proof}
        
        First note that the HJB equation ssociated with $V_i^{\Delta t}$ has a mean field limit $\textbf{w}^{\Delta t}$, as defined by Theorem \ref{thm:mean_limit}, and it satisfies \eqref{eq:small_mean_field_limit}. Following a method similar to Method 2 in Section \ref{sec:MPCdyn}, we assume we have a solution of $\textbf{w}$ and $m$ at time $t$ and we are interested in understanding the solution at a time in the future $t+\Delta t$, for small $\Delta t$. Setting $T = t + \Delta t$ and discretising \eqref{fullmf1} backwards in time we get, up to an error of order $O(\Delta t)$

		\begin{equation}
        	\begin{aligned}
        		\frac{1}{2 \alpha(t + \Delta t)} |D_x \textbf{w}(t + \Delta t,x)|^2 &= \frac{h(x,m_{t + \Delta t})}{\Delta t} + \frac{\textbf{w}(t + \Delta t,x) - \textbf{w}(t,x)}{\Delta t} \\
                &+ f(x,m_{t + \Delta t}) \cdot D_x \textbf{w}(t + \Delta t,x) \\
                &+ \Tr \left( \sigma^2(t + \Delta t,x) D_x^2 \textbf{w}(t + \Delta t,x) \right).
        	\end{aligned}
		\end{equation}
        
        Then, using the terminal condition for $\textbf{w}^{\Delta t}$ we get, up to an error of order $O(\Delta t)$

		\begin{equation}
			\textbf{w}^{\Delta t}(t,x) = (h + \frac{1}{T} g)(x,m_t).
		\end{equation}

		Substituting this into the mean field equation for $m$ gives, up to an error of order $O(\Delta t)$

		\begin{equation}
	    	\begin{aligned}
	    		& \partial_t m_t = D_x \cdot \left[ \left( \frac{1}{\alpha(t)} D_x (h + \frac{1}{T} g)(x,m_t) - f(x,m_t) \right) m_t \right] + \frac{1}{2} \Tr \left( D_x^2 (\sigma^2(t,x) m_t) \right) \\
	            & m_{t= 0} = m_0.
	    	\end{aligned}
		\end{equation}
        
        \end{proof}
        
        \begin{rmk}
            It is interesting to note that the mean field dynamics found here are the same as those for the controlled dynamics in Subsection \ref{sec:MPCdyn}. Thus we can conclude that the best reply strategy for the mean field stochastic differential game can be derived by either first applying the MPC method to the N-player game and then taking the mean field limit, or by first taking the mean field limit of the N-player game and then applying the MPC method.
        \end{rmk}
    
\section{Applications}
    
    In this section we will take some examples from the MFG and BRS literature and use the paradigm of this paper to compare the two approaches.
    
    \subsection{Wealth distribution driven by local Nash equilibria}
    
    	This example is taken from \cite{Degond2014}. Their model described the evolution of agents' wealth and economic configuration (which was noted in \cite{Degond2014} as possibly being a diverse number of attributes, from social status to education level depending on the situation) in time as a response to trading between agents. The trading was assumed to depend on the difference in wealth between two agents that want to trade, and has its origins in \cite{Bouchaud2000}, as well as later work by \cite{During2007}. In this model, we assume $d = 2$, since we are considering agents are described firstly by their wealth and secondly by their economic configuration, so in our framework we have
        
        \[X_i(t) = (Y_i(t),Z_i(t)).\]
        
        Here, $Y_i$ is agent $i$'s economic configuration and $Z_i$ is their wealth. It is assumed that there is no debt in this model, hence $Z_i > 0$ for all $i$. They are governed by the following system of equations
        
        \begin{align}
			& \frac{dY_i}{dt} = v(X_i(t)) \\
            & dZ_i(t) = u_i(t) dt + \sqrt{2d} Z_i(t) dB_i(t).
        \end{align}
        
        Notice that the first equation is deterministic and can not be explicitly controlled, whereas the second equation has a control $u_i$ but no deterministic movement otherwise. $v$ describes the speed at which an agent's economic configuration evolves. Now, we introduce the following notation, similar to notation at the beginning of Section \ref{sec:MPCdyn}, $Y(t) = (Y_i(t))_i$, $Y_{-i}(t) = (Y_j(t))_{j \neq i}$, and similarly for $Z$, $Z_{-i}$. The value functional is given by
        
        \begin{equation}
        	\begin{aligned}
        		&V_i^{\Delta t}(t,x) = \E \left[ \frac{1}{\Delta t}\int_t^{t + \Delta t} \frac{\Delta t u_i^2}{2} + \Phi^{(N)}(X(s)) ds\right] \\
                &X(t) = x, \\ 
                & \Phi^{(N)}(x) =  \frac{1}{N}\sum_{j = 1, \, j \neq i}^N \xi_{i,j}(y) \Psi(|y_i - y_j|) \phi(z_i - z_j).
        	\end{aligned}
        	\label{eq:wealth_cost}
        \end{equation}
        
        In \cite{Degond2014} it is explained that $\phi$ is the trading interaction potential, i.e. it governs the amount of trading that occurs between any two agents based on their difference in wealth. It is also explained that $\xi_{i,j}(Y(s)) \Psi(|Y_i(s) - Y_j(s)|)$ is the trading frequency, i.e. the rate at which trades or movement of wealth takes place between two agents, determined by how far apart the agents' economic configuration is. Several assumptions are made on each of the functions in \ref{eq:wealth_cost}. First, we assume that the function $\phi:\R \to \R$ is $C^2$ and an even function. Second, we assume $\xi_{ij} = \xi_{ji}$ and that it is dependent on the number of other agents in a neighbourhood of the economic configuration of agents $i$ and $j$. As in \cite{Degond2014}, we assume $\xi_{ij}$ has the following form
        
        \begin{equation}
        	\xi_{ij}(y) = \xi\left( \frac{\rho_i^{\Psi} + \rho_j^{\Psi}}{2} \right), \qquad \rho_i^{\Psi} = \sum_{l = 1, \, l \neq i}^N \Psi(|y_j - y_l|).
        \end{equation}
             
        Since each agent can only control their $Y_i$ variable, the HJB equation is modified to 
		
        \begin{equation}
        	\begin{aligned}
        		\frac{(\partial_{z_i}V_i^{\Delta t})^2}{2} = & \frac{\Phi^{(N)}(x)}{\Delta t} + \partial_t V_i^{\Delta t} + v(x_i) \partial_{y_i} V_i^{\Delta t} \\
                &+ \sum_{j = 1 j \neq i}^N (v(x_j) \partial_{y_j} V_i^{\Delta t} + \partial_{z_j} V_j \partial_{z_j} V_i^{\Delta t}) + \sum_{j = 1}^N d z_j^2 \partial_{z_jz_j} V_i^{\Delta t}.
        	\end{aligned}
        \end{equation}
    	
        This is generally an extremely difficult equation to solve and so, although the optimal control is given by $u_i^*(t,x) = \partial_{z_i} V_i^{\Delta t}(t,x)$, the best reply strategy of $u_i^*(t,x) = -\frac{1}{N}\sum_{j = 1, \, j \neq i}^N \xi_{i,j}(y) \Psi(|y_i - y_j|) \phi'(z_i - z_j)$, is a much more tractable and realistic suggestion for how wealth may really be moved.
        
	\subsubsection*{The mean field limit}
        
        Under the framework of this paper, assumptions \textbf{A} and \textbf{C} are automatically fulfilled. It is also relatively straightforward to notice that
        \begin{equation}
        	\Phi^{(N)}(x) = \int_{\R^2} \xi \left( \frac{\rho_N^{\psi}(y_i) - \rho_N^{\psi}(y')}{2} \right) \Psi(|y_i - y'|) \phi(z_i-z') m_{-i}^{N-1}(dx'),
        \end{equation}
        
        where $\rho_N^{\psi}(y_j) = \sum_{l = 1, \, l \neq j}^N \Psi(|y_j - y_l|) = \int_{\R^2} \Psi(|y_j - y'|)m_{-j}^{N-1}(dx)$ for any $j = 1, \ldots, N$ and $m_{-i}^{N-1} = \sum_{j = 1, \, j \neq i}^N \delta_{x_j}$. So, with a trivial modification that doesn't affect the convergence, assumption \textbf{B} is also satisfied. Similarly, we can ensure assumptions \textbf{D}, \textbf{E} and \textbf{F} are satisfied as long as we assume $\Psi$ and $v$ are both Lipschitz. So under these relatively weak assumptions we find that in the mean field limit, with individuals using the BRS, $X_i(t) \to X_t$ for every $i = 1, \ldots, N$ who's distribution evolves according to the following Fokker-Planck equation
        
        \begin{equation}
        	\begin{aligned}
        		\partial_t m_t = - \partial_y(v(x) m_t) + \partial_z(\partial_z\Phi(x,m_t) m_t) + d\partial_{zz}(z^2m_t), \\
                m_{t=0} = m_0.
        	\end{aligned} 
            \label{eq:mfwealth}
        \end{equation}
        
        Here $\Phi(x,m) = \int_{\R^2} \xi \left( \frac{\rho^{\psi}(y,m) - \rho^{\psi}(y',m)}{2} \right) \Psi(|y - y'|) \phi(z-z') m(dx')$ and $ \rho^{\psi}(y,m) = \int_{\R^2} \Psi(|y-y'|)m(dx')$. This is supplemented with various boundary conditions in \cite{Degond2014} to close the PDE problem. Equation \eqref{eq:mfwealth} is also an order $O(\Delta t)$ approximation of the full mean field game below, where $\textbf{w}^{\Delta t}$ is the mean field limit of $V_i^{\Delta t}$, as described in section \ref{sec:mfe}.
        
        \begin{align}
        	&\begin{aligned}
        		&\partial_t m_t = - \partial_y(S(x) m_t) + \partial_z(\partial_z\textbf{w}(t,x) m_t) + d\partial_{zz}(z^2m_t), \\
                &m_{t=0} = m_0.
        	\end{aligned} \\
            	&\begin{aligned}
            		&(\partial_y \textbf{w}^{\Delta t}(t,x))^2 + (\partial_z \textbf{w}^{\Delta t}(t,x))^2 = \frac{\Phi(x,m_t)}{\Delta t} + \partial_t \textbf{w}^{\Delta t}(t,x) + S(x) \partial_y \textbf{w}^{\Delta t}(t,x) + 2dz^2 \partial_{zz}\textbf{w}^{\Delta t}(t,x) \\
                    &\textbf{w}(T,x) = 0.
            	\end{aligned}
        \end{align}
        
         It is clear that either solving \eqref{eq:mfwealth} numerically, or analytically showing that solutions do exist is a much simpler problem than solving the full mean field equations related to the fully optimal solution.
        
	\subsection{Congestion and aversion in pedestrian crowds}
    
    This second example has been taken from \cite{Lachapelle2011}. In the paper, the authors begin with an overview of the different methods for modelling traffic and pedestrian dynamics, followed by a description of how mean field games may be used as a bridge from microscopic traffic models to macroscopic. The paper then continues by describing an MFG model of pedestrian traffic. This model is perfectly suited to adapting to a BRS approach, firstly because the cost function implemented by \cite{Lachapelle2011} can be adapted to the approach taken in this paper, and secondly because it is natural to imagine that individuals in a crowd don't optimise their own behaviour based on the long-term future behaviour of other individuals around them, as described by the complex MFG framework. Rather, an assumption that individuals look at the flow around them at an almost instantaneous moment in time and change their behaviour accordingly seems to fit more naturally to our lived experience and is best described through the BRS framework.
    
    The paper \cite{Lachapelle2011} considers two populations of groups, their analysis begins with assuming the mean field limit has been taken and that in this limit, the distribution of each group is absolutely continuous with respect to the Lebesgue measure. We can modify our analysis from Section \ref{sec:MPC} to accomodate these ideas. We begin by considering two populations of individuals, with distribution functions $m_1(t,x)$ and $m_2(t,x)$ respectively. The respective positions in space $\R^d$ of a representative particle are given by $Y(t)$ and $Z(t)$. They move according to the following SDE
    
    \begin{align}
    	dY(t) = \alpha(t) dt + \sigma dB_1(t) \label{eq:pedmv1}\\
        dZ(t) = \beta(t) dt + \sigma dB_2(t). \label{eq:pedmv2}
    \end{align}
    
    In \eqref{eq:pedmv1}--\eqref{eq:pedmv2}, $\alpha$ and $\beta$ are the controls of the populations, $\sigma \in \diag(\R^d)$ is a diagonal positive matrix, and $B_1$ and $B_2$ are independent $d$-dimensional Brownian motions. As in \cite{Lachapelle2011}, we focus on the two populations interacting on some domain $\Omega \subset \R^2$. The cost function being optimised by each representative player is given by
    
    \begin{align}
    	V_{\lambda}(t,m_1,m_2) = \E \left[ \int_t^T \left( \frac{|\alpha(s)|^2}{2} + \Phi_{\lambda}(Y(s),m_1(s),m_2(s)) \right) ds + \Psi_1(Y(T)) \right] \label{eq:pedcost1} \\
        W_{\lambda}(t,m_y,m_z) = \E \left[ \int_t^T \left( \frac{|\beta(s)|^2}{2} + \Phi_{\lambda}(Z(s),m_2(s),m_1(s)) \right) ds + \Psi_2(Z(T)) \right] \label{eq:pedcost2} \\
        (Y(t),Z(t)) \sim (m_1(t),m_2(t)). \label{eq:pedcost3}
    \end{align}
    
    Using the formulation of $\Phi$ in \cite{Lachapelle2011}, it can be consistently defined to match the framework of this paper by
    
    \begin{equation}
    	\Phi_{\lambda}(x,m_i,m_j) = m_i(x) + \lambda m_j(x), \qquad \lambda > 0.
    \end{equation}
    
    Here, \cite{Lachapelle2011} described $\lambda$ as the 'xenophobia parameter', that is it measures how averse each group is to one another. If $\lambda$ is high then the two groups will separate as much as possible, whereas if $\lambda$ is low, the groups will be as or more worried about their distance between individuals in the same group than those in the opposite group. Equations \eqref{eq:pedmv1} --- \eqref{eq:pedcost3} are formulated in the following mean field game system in \cite{Lachapelle2011}
    
    \begin{align}
    	\partial_t m_1 = \frac{1}{2} \Tr \left( \sigma^2 D_x^2 m_1 \right) +  D_x \cdot (D_x V_{\lambda} m_1) \\
        \partial_t m_2 = \frac{1}{2} \Tr \left( \sigma^2 D_x^2 m_2 \right) +  D_x \cdot (D_x W_{\lambda} m_2) \\
       \frac{|D_x V_{\lambda}|^2}{2} = \partial_t V_{\lambda} + \frac{1}{2} \Tr \left( \sigma^2 D_x^2 V_{\lambda} \right) + \Phi(x,m_1,m_2) \\
        \frac{|D_x W_{\lambda}|^2}{2} = \partial_t W_{\lambda} + \frac{1}{2} \Tr \left( \sigma^2 D_x^2 W_{\lambda} \right) + \Phi(x,m_2,m_1).
    \end{align}
    
    Under the paradigm of this paper, we consider that in fact in each interval $[t,t + \Delta t]$ agents are minimising the following cost with respect to a fixed (in time) control random variable.
    
    \begin{align}
    	V_{\lambda}^{\Delta t}(t,m_1,m_2) = \E \left[ \int_t^{t + \Delta t} \left( \frac{|\alpha(s)|^2}{2} + \frac{1}{\Delta t} \Phi_{\lambda}(Y(s),m_1(s),m_2(s)) \right) ds + \frac{1}{T} \Psi_1(Y(t + \Delta t)) \right] \label{eq:mpcpedcost1} \\
        W_{\lambda}^{\Delta t}(t,m_y,m_z) = \E \left[ \int_t^{t + \Delta t} \left( \frac{|\beta(s)|^2}{2} + \frac{1}{\Delta t} \Phi_{\lambda}(Z(s),m_2(s),m_1(s)) \right) ds + \frac{1}{T} \Psi_2(Z(t + \Delta t)) \right] \label{eq:mpcpedcost2} \\
        (Y(t),Z(t)) \sim (m_1(t),m_2(t)). \label{eq:mpcpedcost3}
    \end{align}    
    
    Using the best reply strategy approach, we are able to simplify this system to the following two equations, which describe the evolution of the two populations.
    
    \begin{align}
    	\partial_t m_1 = \frac{1}{2} \Tr \left( \sigma^2 D^2 m_1 \right) +  D \cdot ((D m_1 + \lambda D m_2 + \frac{1}{T} D \Psi_1) m_1) \\
        \partial_t m_1 = \frac{1}{2} \Tr \left( \sigma^2 D^2 m_2 \right) + D \cdot ((D m_2 + \lambda D m_1 + \frac{1}{T} D \Psi_2) m_1).
    \end{align}
    
    This section has clearly shown some of the potential benefits of using the BRS to replace MFGs in certain situations, of course when exactly this is appropriate requires further investigation. However, it is intuitive from the formulation of the BRS that in situations where short time horizons are considered and agents are unable to optimise their behaviour efficiently then there is a case for using the BRS.
    
    \section{Conclusion and future perspectives}
    
    In conclusion, we have shown that the BRS, a sub-optimal strategy for players in a stochastic differential game, can be derived from the optimal strategy as an asymptotic limit of a revised cost functional. The BRS is an important alternative strategy to the MFG to consider when the time horizon of the optimisation problem is small because it depends only on the running and terminal costs. As a result there is no HJB equation to solve, and since the HJB equation is often intractable the BRS offers a more tractable modelling approach and at reduced computational effort. We then showed how, under certain conditions, the BRS can produce a mean field limit as the number of players tends to infinity. To close our analysis, we  proved that the mean field game equations first introduced by Lasry and Lions \cite{Lasry2007}, which are the mean field limit of the stochastic differential game, can also be approximated by the BRS. We concluded that regardless of whether we approximate the MFG by the mean field BRS, or approximate the $N$-player stochastic differential game by the BRS first and then take the mean field limit, the resulting dynamics of the distribution of players is the same.
    
    In the final section we were able to analyse two examples from existing literature. In the first, the BRS was already used as the dynamics for the mean field behaviour, so we can now justify this use by explaining that the agents involved in the behaviour approximately minimise a related cost. In the second example, we show how a mean field game for congestion could be approximated by using the BRS. This simplified the behaviour considerably and could allow us to computationally model the behaviour more efficiently. We have a number of future directions.
    
    Throughout the paper we have had to renormalise the optimisation problem to obtain the BRS as an approximation to a solution to the game. We have not claimed that the resulting BRS will now approximate the MFG solution for the original optimisation problem. In fact, one can imagine situations where the BRS will be qualitatively similar to the MFG and situations where they won't. We hope to explore more direct comparisons between MFG and BRS dynamics in future work.
    
\appendix

\section*{Appendix}

\section{Differentiability in the space of measures} \label{sec:meas}

	In this section, we will outline some key definitions and results about the space of measures, and in particular differentiability in the space of measures. This material is not new and is sourced from \cite{Delarue2018,Cardaliaguet2015a,Ambrosio2005,Cardaliaguet2010}. Throughout this section, $(E,d)$ will be any complete separable metric space, $\mu$ and $\nu$ will be probability measures on $(E,d)$ and $\Gamma(\mu,\nu) := \{ \gamma \in \mathcal{P}(E^2) : \forall A \in \mathcal{B}(E), \, \gamma(A \times E) = \mu(A) \text{ and } \gamma(E \times A) = \nu(A) \}$, where $\mathcal{P}(E)$ is the set of probability measures on $E$ and $\mathcal{B}(E)$ is the Borel sigma algebra on E.
    
    \vspace{\parskip}
    
    \begin{defn} \label{def:Ws}
    	For any $p \geq 1$, the set of probability measures of order $p$, denoted by $\mathcal{P}_p(E) \subset \mathcal{P}(E)$, is the set of $\mu \in \mathcal{P}(E)$ such that for any $y \in E$
        
        \[\int_E (d(x,y))^p \mu(dx) < \infty.\]
    \end{defn}
    
    \vspace{\parskip}
    
    \begin{defn} \label{def:Wd}
    	Let $\mu, \nu \in \mathcal{P}_p(E)$ and let $p \geq 1$. The $p$-Wasserstein distance between $\mu$ and $\nu$, denoted by $\mathcal{W}_p(\mu,\nu)$ is defined as
        
        \[\mathcal{W}_p(\mu,\nu) = \inf_{\gamma \in \Gamma(\mu,\nu)} \left\{ \left[ \int_{E^2} (d(x,y))^p \gamma(dx,dy) \right]^{1/p} \right\}.\]
        
	\end{defn}
        
	With this distance defined, $\mathcal{P}_p(E)$ is a metric space. Definitions \ref{def:Ws} and \ref{def:Wd} can be equivalently reformulated in terms of random variables as follows. We now let $(\Omega, \mathcal{F},\mathbb{P})$ be an atomless probability space.
    
    \vspace{\parskip}
    \begin{rmk}
        Definitions \ref{def:Ws} and \ref{def:Wd} can be restated in a complementary form using random variables. For example $\mathcal{P}_p(E)$ can be defined as the set of $\mu \in \mathcal{P}(E)$ such that for any $y \in E$ and any random variable $X:\Omega \to E$ with law $\mathcal{L}(X) = \mu$
        
        \[\E[(d(X,y))^p] < \infty\].
        
      Similarly, $\mathcal{W}_p(\mu,\nu)$ can be defined as
        
        \[\mathcal{W}_p(\mu,\nu) = \inf \left\{ \E [(d(X,Y))^p]^{1/p} : \mathcal{L}(X) = \mu, \, \mathcal{L}(Y) = \nu	\right\}.\]
    \end{rmk}
    
    The reason for defining $\mathcal{P}_p(E)$ and $\mathcal{W}_p$ in this alternative way is that it is often easier to work with random variables than it is to work with measures directly. It is important to note that since $E$ is a Polish space (complete and separable) and $\Omega$ is atomless, it is always possible to find such an $X$ with law $\mathcal{L}(X) = \mu$. Now that we have put a metric space structure onto the set of probability measures, the next step is to define differentiability of functions with respect to measure. There are several overlapping, but not equivalent, definitions of differentiability in the space of measures. In this appendix we will discuss L-differentiability, as defined in \cite{Delarue2018}, and the notion of a functional derivative described in \cite{Cardaliaguet2015a}. There is also another definition of differentiability, defined in \cite{Ambrosio2005}, which is in some sense a more intrinsic definition of differentiability however it is less useful to us here and so will be ignored in this appendix. For this section of the appendix, as with the main body of the paper, we will restrict our analysis to focussing on $\mathcal{P}_2(\R^d)$. We will focus on functions $u: \mathcal{P}_2(\R^d) \to \R$ and consider their lift $\tilde{u}:L^2((\Omega, \mathcal{F}, \mathbb{P});\R^d) \to \R$ defined by $\tilde{u}(X) = u(\mathcal{L}(X))$. As previously discussed, it is always possible to find such a random variable $X$ given a measure $\mu$.
    
    \vspace{\parskip}
    
    \begin{defn} \label{def:Ldiff}
    	Let $u$ and $\tilde{u}$ be as defined above. $u$ is (continuously) L-differentiable at $\mu \in \mathcal{P}_2(\R^d)$ if there exists a random variable $X \in L^2((\Omega, \mathcal{F}, \mathbb{P});\R^d)$ such that $\tilde{u}$ is differentiable in the usual Fréchet sense at $X$ (or continuously differentiable in an open neighbourhood of $X$ in the case of continuously L-differentiable). Note that in this case, we consider $D\tilde{u}(X) \in L^2((\Omega, \mathcal{F}, \mathbb{P});\R^d)$ by associating this space with its dual. 
    \end{defn}
    
    It is not clear straight away that the above definition of differentiability is independent of the choice of $X$, however the next two propositions (to be found, along with proofs, in \cite{Delarue2018}) show that this is indeed the case and, under certain circumstances the derivative can be uniquely described by the measure $\mu$.
    
    \vspace{\parskip}
    
    \begin{prop} \label{prop:Ldiff}
    Let $u$ and $\tilde{u}$ be as defined previously. Suppose $u$ is L-differentiable at $\mu \in \mathcal{P}_2(\R^d)$. Then, for all $X \in L^2((\Omega, \mathcal{F}, \mathbb{P});\R^d)$ such that $\mathcal{L}(X) = \mu$, $\tilde{u}$ is differentiable at $X$ and the law of $(X,D\tilde{u}(X))$ is independent of the $X$ chosen so that $\mathcal{L}(X) = \mu$.
    \end{prop}
     
    \vspace{\parskip}
    
    \begin{prop} \label{prop:cLdiff}
    Let $u$ and $\tilde{u}$ be as defined previously. Suppose $u$ is everywhere continuously L-differentiable. Then, for every $\mu \in \mathcal{P}_2(\R^d)$, there exists a deterministic, measurable function $\xi:\R^d \to \R^d$ such that for all $X \in L^2((\Omega, \mathcal{F}, \mathbb{P});\R^d)$ with $\mathcal{L}(X) = \mu$ we have $D\tilde{u}(X)(\omega) = \xi(X(\omega))$ for almost every $\omega \in \Omega$.
    \end{prop}
    
    The importance of these two propositions, as explained in \cite{Delarue2018}, are as follows. Proposition \ref{prop:Ldiff} means that differentiability with respect to a measure $\mu$ depends only on $\mu$ and not on the particular random variable chosen to represent the derivative. Proposition \ref{prop:cLdiff} states that if there is some further regularity in the differentiability, then not only is the L-derivative independent of the random variable, it is of the form $\xi(X)$ for some deterministic function $\xi$ which is uniquely defined almost everywhere. Due to this uniqueness property, we can then define the L-derivative of $u$ as follows
    
    \vspace{\parskip}
    
    \begin{defn}
    	 Let $u$ and $\tilde{u}$ be as defined previously. Suppose $u$ is continuously L-differentiable and $\xi$ is as in Proposition \ref{prop:cLdiff}. Then, the L-derivative of $u$ at $\mu \in \mathcal{P}_2(\R^d)$, denoted by $\partial_m u(\mu)$ is defined as the equivalence class of $\xi$ in $L^2((\R^d,\mu);\R^d)$. This is defined uniquely since $\xi$ is defined uniquely almost everywhere with respect to $\mu$.
    \end{defn}
    
    Note that since $\partial_m u(\mu)$ is an equivalence class of functions from $\R^d$ to $\R^d$, it can be identified with a function $\partial_m u(\mu)(\cdot): \R^d \to \R^d$ without ambiguity. We shall often consider $\partial_mu(\mu)$ in such a way without explicit reference to this form. As mentioned near the beginning of this appendix, the notion of the functional derivative of a function with respect to a measure will also be a widely used notion for us in this paper. We will now define what this notion is and link it to the previous definition of the L-derivative. The following definition is attributed to \cite{Cardaliaguet2015a}.
    
    \vspace{\parskip}
    
    \begin{defn}
    	Let $u:\mathcal{P}_2(\R^d) \to \R^d$. We call $u$ a $\mathcal{C}^1$ function if there exists some function $\frac{\delta u}{\delta m}: \mathcal{P}_2(\R^d) \times \R^d \to \R^d$ such that for all $\mu, \nu \in \mathcal{P}_2(\R^d)$ the following holds.
        
        \begin{equation} \label{eq:fundiv}
        	\lim_{s \to 0^+} \frac{u((1-s) \mu + s \nu) - u(\mu)}{s} = \int_{\R^d} \frac{\delta u}{\delta m}(\mu)(y) (\nu - \mu)(dy).
        \end{equation}

        Here, $\frac{\delta u}{\delta m}(\mu)$ is defined up to a constant, so the normalisation condition $\int_{\R^d} \frac{\delta u}{\delta m}(\mu)(y) \mu(dy) = 0$ is taken. The requirement \eqref{eq:fundiv} can easily be seen as equivalent to the following requirement, which is the requirement used in Section \ref{sec:mfe}.
    
    \[u(\nu) - u(\mu) = \int_0^1 \int_{\R^d} \frac{\delta u}{\delta m}(s\nu - (1-s) \mu)(y) (\nu - \mu)(dy) ds.\]
    
    \end{defn}
    
    These two notions of derivative have a simple relationship to each other, as explained in Propositions \ref{prop:Dequiv1} and \ref{prop:Dequiv2} below (see \cite{Cardaliaguet2015a} and \cite{Delarue2018} respectively for original statements and proofs).
    
    \vspace{\parskip}
    
    \begin{prop} \label{prop:Dequiv1}
    	Let $u:\mathcal{P}_2(\R^d) \to \R$ be $\mathcal{C}^1$. Assume further that the function $\frac{\delta u}{\delta m}(\mu)(\cdot):\R^d \to \R^d$ is continuously differentiable for any $\mu \in \R^d$. Then $u$ is L-differentiable and we have
        
        \begin{equation} \label{eq:measure_equivalence}
        	\partial_m u(\mu)(x) = D_x \frac{\delta u}{\delta m}(\mu)(x).
        \end{equation}

    \end{prop}
    
    \vspace{\parskip}
    
    \begin{prop} \label{prop:Dequiv2}
    	Let $u:\mathcal{P}_2(\R^d) \to \R$ be L-differentiable. Assume further that the Fréchet derivative of $\tilde{u}$ is Lipschitz and that for all $\mu \in \mathcal{P}_2(\R^d)$ there is a representative $\partial_m u(\mu)(\cdot)$ such that $\partial_m u: \mathcal{P}_2(\R^d) \times \R^d \to \R^d$ is continuous. Then $u$ is $\mathcal{C}^1$ and satisfies \eqref{eq:measure_equivalence}

    \end{prop}
    
    In Section \ref{sec:MFGMPC} and in particular Section \ref{sec:mfe}, we rely heavily on this functional derivative notion, and implicitly use Propositions \ref{prop:Dequiv1} and \ref{prop:Dequiv2} to interchange the definitions.

\end{document}